\documentclass[a4paper, 11pt]{amsart}

\usepackage{enumerate}
\usepackage{mathrsfs}
\usepackage{amssymb}
\usepackage[all]{xy}
\usepackage{MnSymbol}

\title[]
{Global branching laws by global Okounkov bodies}
\author{Henrik Sepp\"anen*} 

\thanks{*supported by the DFG Priority Programme 1388 ``Representation 
Theory"}

\subjclass[2000]{}
\date{\today}

\address{Henrik Sepp\"{a}nen,
Mathematisches Institut,
Georg-August-Universit\"at G\"ottingen,
Bunsenstra\ss e 3-5, 
D-37073 G\"ottingen,
Germany}
\email{hseppaen@uni-math.gwdg.de}

\newcommand{\C}{\mathbb{C}}
\newcommand{\R}{\mathbb{R}}
\newcommand{\N}{\mathbb{N}}
\newcommand{\Z}{\mathbb{Z}}
\newcommand{\Q}{\mathbb{Q}}

\newcommand{\Oh}{\mathcal{O}}

\newtheorem{prop}{Proposition}[section]
\newtheorem{lemma}[prop]{Lemma}

\newtheorem{thm}[prop]{Theorem}
\theoremstyle{definition}
\newtheorem{defin}[prop]{Definition}
\newtheorem{rem}[prop]{Remark}

\begin{document}

\maketitle

\begin{abstract}
Let $G'$ be a complex semisimple group, and let $G \subseteq G'$ be 
a semisimple subgroup. We show that the branching cone of the pair $(G, G')$, 
which (asymptotically) parametrizes all pairs $(W, V)$ of irreducible finite-dimensional  
$G$-representations $W$ which occur as subrepresentations of a finite-dimensional 
irreducible $G'$-representation $V$, can be identified with the pseudo-effective 
cone, $\overline{\mbox{Eff}}(Y)$, of some GIT quotient $Y$ of the flag variety of 
the group $G \times G'$. Moreover, we prove that the quotient $Y$ is a Mori dream 
space.

As a consequence, the global Okounkov body $\Delta(Y)$ of $Y$, with respect to 
some admissible flag of subvarieties of $Y$, is fibred over the branching 
cone of $(G, G')$, and the fibre $\Delta(Y)_{(W, V)}$ over a point $(W, V)$ 
carries information about (the asymptotics of) the multiplicity of $W$ in $V$.
Using the global Okounkov body $\Delta(Y)$, we easily derive a 
multi-dimensional generalization of Okounkov's result about the log-concavity 
of asymptotic multiplicities. 
\end{abstract}

\section{Introduction}

Let $G'$ be a complex semisimple algebraic group, and let $G \subseteq G$ be 
a complex semisimple subgroup. 
We are interested in decomposing finite dimensional irreducible 
representations $V$ of $G'$ into irreducible $G$-representations; 
\begin{align*}
V=\bigoplus_j m_jW_j,
\end{align*}
where $W_j$ is an irreducible $G$-representation and 
$m_j=\mbox{dim}(\mbox{Hom}_G(W_j, V))$ is its multiplicity 
in $V$. 

By the Borel-Weil Theorem, we can reformulate the 
problem into a question concerning sections of line 
bundles as follows. Each finite-dimensional irreducible 
$G'$-representation can be realized geometrically 
as the space\\ $H^0(G'/B', L_\lambda)$ of all sections 
of a line bundle $L_\lambda$ over the flag variety $G'/B'$, 
where $B'$ is a Borel subgroup of $G'$, and the parameter 
$\lambda$ is a dominant weight with respect to a 
maximal torus $T' \subseteq B'$ and a given choice of 
positive roots. Likewise, the irreducible 
$G$-representations are realized as spaces of 
sections of line bundles over $G$-flag varieties, so 
that the representation-theoretic problem of determining 
subspaces of $G$-invariants in the tensor products $W_j^* \otimes V$ 
amounts to the geometric problem of determining the $G$-invariant sections 
in the spaces $H^0(G/B \otimes G'/B', L_\mu \otimes L_\lambda)$, 
where $B \subseteq G$ is a Borel subgroup of $G$, and $\mu$ is a dominant weight 
with respect to a maximal torus $T \subseteq G$ and 
a choice of positive roots for the root system of $G$. In other words, we 
are interested in $G$-invariant sections of line bundles $L$ over 
the flag variety $X:=(G \times G')/(B \times B')$.
%The $G$-equivariant 
%homomorphisms $\varphi: H^0(G/B, L_\nu) \rightarrow H^0(G'/B', L_\lambda)$ 
%correspond to $G$-invariant sections of the 
%line bundle  $\overline{L^*_\mu} \boxtimes L_\lambda \rightarrow \overline{G/B} \times G'/B'$.
%Here $ \overline{G/B}$ denotes the smooth manifold $G/B$ equipped with 
%the opposite complex structure, i.e., the structure sheaf of  $\overline{G/B}$ 
%is given by the complex-conjugate of the structure sheaf of $G/B$. 
%The line bundle $\overline{L^*_\mu}$ is the $\C$-anti-linear dual of the 
%line bundle $L_\mu$. 

%By replacing the Borel subgroup $B$ by its image under the longest element of the 
%Weyl group of $G$, we may thus reformulate 
%the decomposition problem as the question of describing the 
%$G$-invariant sections  of line bundles $L$ over $X:=G/B \times G'/B'$.

If $L$ is ample we can form the GIT quotient $X^{ss}(L)//G$. Then there 
exists a $q \in \N$ and a  line bundle $L_0$ over $X^{ss}(L)//G$, 
such that we have isomorphisms 
\begin{align*}
H^0(X, L^{kq})^G \cong H^0(X^{ss}(L)//G, L_0^k)
\end{align*}   
for all $k \in \N$. In other words, we have translated the problem of 
finding $G$-invariant sections of powers of $L$ into the problem of studying the 
full linear series of a line bundle $L_0$ on the GIT-quotient $X^{ss}(L)//G$. 

Note that the quotient $X^{ss}(L)//G$ depends on the line bundle $L$. 
It would be desirable to be able to simultaneously study the $G$-invariant sections 
of \emph{all} line bundles over $X$ in terms of understanding all sections of 
all line bundles over some variety $Y$. 

In this note we use VGIT to construct a ``universal quotient", $Y$, of $X$ by $G$, 
yielding a geometric formulation of the problem of finding all branching 
laws for the pair $(G, G')$ in the following sense:
For every line bundle $L$ on $X$, there exists a $q \in \N$ and a line bundle 
$L_0$ on $Y$, such that 
\begin{align*}
H^0(X, L^{kq})^G \cong H^0(Y, L_0^k), \quad k \in \N.
\end{align*}  

In order to explain our main results in more detail, we recall that the 
\emph{branching cone} $\Gamma(G, G')$ of the 
pair $(G, G')$ is the closed convex cone (in the direct sum of two given 
Cartan subalgebras of the Lie algebras of $G$ and $G'$, respectively) 
generated by all pairs $(\mu, \lambda)$, where $\lambda$ is a dominant weight 
for $G'$, and $\mu$ is a dominant weight for $G$, such that 
the corresponding irreducible $G$-representation occurs in the 
irreducible $G'$-representation of highest weight $\lambda$.  
If a rational point $(\mu, \lambda)$ lies in $\Gamma(G, G')$ we 
thus know that for some $k \in \N$ the irredducible 
$G$-representation $W_\mu$ occurs as a $G$-subrepresentation in the 
irreducible $G'$-representation $V_\lambda$. However, $\Gamma(G, G')$ gives 
no information about the multiplicity of $W_\mu$ in $V_\lambda$. 

In this paper, we show how to contruct a cone which is fibred over 
$\Gamma(G, G')$, and whose fibres above points in $\Gamma(G, G')$ 
describe asymptotic multiplicities. We collect the main resuls 
of the paper in the following theorem.

\begin{thm}
There exists a cone $\Delta=\Delta_Y$, depending on the universal quotient $Y$, and 
a surjective linear map $p: \Delta \rightarrow \Gamma(G, G')$, such that\\

\noindent (i) For each rational point $(\mu, \lambda)$ in the interior of 
$\Gamma(G, G')$, 
\begin{align*}
\lim_{k \rightarrow \infty} \frac{\mbox{dim} \mbox{Hom}_G(W_{k \mu}, V_{k \lambda})}{k^n}
=\mbox{Vol}_n(p^{-1}(\mu, \lambda)),
\end{align*}
where $n$ is the dimension of the quotient $Y$, and the right hand side denotes the $n$-dimensional volume 
of the fibre $p^{-1}(\mu, \lambda) \subseteq \Delta_Y$.\\

\noindent (ii) The branching cone $\Gamma(G, G')$ is linearly isomorphic to the presudo-effective cone 
$\overline{\mbox{Eff}}(Y)$ of $Y$, and the cone $\Delta_Y$ is the global Okounkov body of $Y$ with 
respect to some admissible flag of subvarieties of $Y$.\\

\noindent (iii) The quotient $Y$ is a Mori dream space.

\end{thm}

We would like to end this introduction with a brief history of the branching cone.
It was proven by Brion and Knop (cf. \cite{E}) that the semigroup of all pairs 
$(\mu, \lambda)$ of dominant weights which generate the branching cone 
is a finitely generated semigroup of the group of all integral weights 
of the respective Cartan subalgebra of the Lie algebra 
$\mathfrak{g} \oplus \mathfrak{g'}$ of $G \times G'$. 
Berenstein and Sjamaar (\cite{BS}) obtained a list of inequalities determining the branching 
cone, which, however, is redundant. Later, Ressayre (\cite{R10}) obtained a minimal list of 
inequalities defining this cone, i.e., a list where none of the inequalities is redundant. 
We would like to remark that Ressayre also used the identification of the branching 
cone with the  $G$-ample cone of the flag variety $X=G/B \times G'/B'$. 
For a more thorough history of the branching cone, involving special cases, we 
refer to \cite{R10}. 

As a final remark on branching cones, we point out that 
in general $G'$ and $G$ are only assumed to be reductive, and not semisimple, whereas
we have chosen the more restrictive semisimple setting. The reason is not merely 
technical, as in the use of the (essential) uniqueness of the moment map 
in Section 3, but also since we do not expect the result on the 
identification of the branching cone with a pseudo-effective cone to 
hold in the more general setting (cf. Remark \eqref{R: coneid}).

Finally, we would like to mention a couple of other approaches to 
the branching problem. First of all, instead of studying $G$-invariant sections 
of line bundles over the product $X$, one could study 
sections of line bundles over $G'/B'$, or indeed any projective 
$G$-variety, which are invariant under a maximal unipotent subgroup 
$U$ of $G$, and consider the Okounkov bodies defined using only 
$U$-invariants. This was the original approach in Okounkov's 
famous paper (\cite{Ok96}), and it was later generalized 
by Kaveh and Khovanskii (\cite{KK10}). (Our main reason 
for passing to quotients is that we want to get rid of 
all conditions of invariance under any subgroup in order 
to be able to use generic arguments, {\it e.g.} Bertini-type 
arguments (cf. Remark \ref{R: shapeok}).)

Secondly, C. Manon (\cite{M1}, \cite{M2}) has also addressed the problem 
of constructing cones above the branching cones 
using Okounkov bodies. However, these Okounkov bodies are 
of a more combinatorial nature than in our approach.\\

The paper is organized as follows. In Section 2 we introduce the setting and recall some 
results about the variation of GIT quotients; VGIT. In Section 3 we study the descent 
of line bundles over the flag variety $X$ to various GIT quotients of $X$. 
In Section 4 we contruct a GIT quotient $Y$ which is a Mori dream space, and 
we identify the pseudo-effective cone, $\overline{\mbox{Eff}}(Y)$, of $Y$ with 
the $G$-ample cone $C^G(X)$, i.e., with the branching cone for the 
pair $(G, G')$. Finally, in Section 5 we use global Okounkov bodies of $Y$ to 
study asymptotics of the function 
$k \mapsto \mbox{dim}\, \mbox{Hom}_G(W_{k\mu}, V_{k\lambda}), \, k \in \N$, 
for a fixed pair $(\mu, \lambda)$ in the branching cone of $(G, G')$. 
Using the global Okounkov body of $Y$, we easily derive a multi-dimensional generalization 
of a log-concavity result for multiplicities due to Okounkov (\cite{Ok96}).

{\bf Acknowledgement:} I would like to thank M. Brion for 
interesting discussions about GIT during the preparation of this 
paper, as well as for helpful remarks on a preliminary version. I am also 
grateful to V. Tsanov for comments. 

\section{Preliminaries}
Let $T \subseteq G$, $T' \subseteq G'$, and 
$B \subseteq G$ and $B' \subseteq G'$ be 
maximal tori, and Borel subgroups of $G'$ and $G$, respectively, such that 
$T \subseteq B$ and $T' \subseteq B'$. Let $\mathfrak{g}, \mathfrak{g}', \mathfrak{b}, \mathfrak{b}', 
\mathfrak{t}, \mathfrak{t}'$be the Lie algebras of the groups $G, G', B, B', T$, and $T'$. 
Fix a choice of roots $\mathcal{R}^+$ for the root system on $\mathfrak{g} \oplus \mathfrak{g}'$ with 
respect to the Cartan subalgebra $\mathfrak{t} \oplus \mathfrak{t}'$ such that 
the Borel subalgebra $\mathfrak{b} \oplus \mathfrak{b}'$ is the direct sum of 
$\mathfrak{t} \oplus \mathfrak{t}'$ and 
the sum of the negative root spaces, 
\begin{align*}
\mathfrak{b} \oplus \mathfrak{b}'=\mathfrak{t} \oplus \mathfrak{t}' \oplus \bigoplus_{\alpha \in -\mathcal{R}^+}
(\mathfrak{g} \oplus \mathfrak{g}')_\alpha.
\end{align*}
Let $X:=G/B \times G'/B'$ be 
the product of the corresponding flag varieties. Then $G$ acts on $X$ by 
the diagonal action
\begin{align*}
(f, (gB, hB')):=(fgB,fhB' ), \quad f, g \in G, \, h \in G'.
\end{align*}

Now, let $U \subseteq G$ and $U' \subseteq G'$ be maximal compact subgroups, 
such that the complex maximal tori $T$ and $T'$ are complexifications 
of maximal tori $T_\R \subseteq U$ and $T'_\R \subseteq U'$, respectively. 
Then the flag variety $X$ is naturally isomorphic as a complex manifold to the 
 quotient $(U \times U')/(T_\R \times T'_\R)$, and can thus also be realized 
as a coadjoint $(U \times U')$-orbit in the dual of Lie algebra 
$\mathfrak{u} \oplus \mathfrak{u}'$ of $U \times U'$.  

For a line bundle $L \rightarrow X$, and a section $s \in H^0(X, L)$, let 
$$X_s:=\{x \in X \mid s(x) \neq 0\}.$$
The set of \emph{semi-stable points} of $X$ with respect to $L$ is 
the set
\begin{align*}
X^{ss}(L):=\{x \in X \mid \exists m \in \N \, \exists s \in H^0(X, L^m) \,\,\mbox{such that}\, s(x) \neq 0\}.
\end{align*}
The set of \emph{stable} points of $X$ with respect to $L$ is the set
\begin{align*}
X^s(L)=\{x \in X^{ss}(L) \mid \mbox{the orbit}\, G.x \, \mbox{is closed in} \, X^{ss}(L) \, 
\mbox{and} \, G_x \,\mbox{is finite} \}
\end{align*}

\noindent Let $N^1(X) \cong \Z^{\mbox{dim}(T \times T')}$  be the N\'eron-Severi group of 
$X$, and let $N^1(X)_\R:=N^1(X) \otimes_\Z \R$. We recall the following definitions (cf. \cite{DH}).
\begin{defin}
(i) A line bundle $L \rightarrow X$ is \emph{$G$-effective},  if $X^{ss}(L) \neq \emptyset$ and 
$X_s \subseteq X$ is affine.\\

(ii) The {$G$-ample cone} $C^G(X)$ is the closed convex cone in $N^1(X)_\R$ generated by 
all $G$-effective ample line bundles. 
\end{defin}

\section{Descent of line bundles to quotients}

In this section we study the descent of line bundles to GIT quotients of $X$, in particular
its dependence on the location of the line bundle used to define semi-stability in the $G$-ample cone; 
more precisely, in terms of the \emph{GIT chambers} in $C^G(X)$. 
In doing so, although we have otherwise taken the algebro-geometric approach, we will use transcendental 
methods, notably the moment maps of the manifold $X$ with respect to various symplectic forms and 
the action of a maximal compact subgroup of $G$. In the following sections we will, however, only 
be interested in quotients defined by a \emph{chamber} in $C^G(X)$, as in Theorem \ref{T: descent} ii), 
and for this purpose the algebro-geometric approach is sufficient.\\

Let $\mathfrak{t}$ and $\mathfrak{t}'$ be the Lie algebras of $T$ and $T'$, respectively, 
and let $\mathcal{P}^+(\mathfrak{t} \oplus \mathfrak{t}')$ denote the set of 
dominant integral weights of the Lie algebra $\mathfrak{t} \oplus \mathfrak{t}'$
with respect to the positive system $\Phi^+$.

The \emph{Cox ring} $\mbox{Cox}(X)$, or \emph{total coordinate ring}, of $X$ 
is then the $\mathcal{P}^+(\mathfrak{t} \oplus \mathfrak{t})$-graded 
ring 
\begin{align*}
\mbox{Cox}(X)=\bigoplus_{\nu \in \mathcal{P}^+(\mathfrak{t} \oplus \mathfrak{t}) }
H^0(X, L_\nu),
\end{align*}
where $L_\nu:=(G \times G') \times_{\nu} \C$ is the 
line bundle on $X$ induced from the character of $B \times B'$ defined by 
$\nu \in (\mathfrak{t} \oplus \mathfrak{t})^*$. 
Since the multiplication map
\begin{align*}
H^0(X, L_\nu) \otimes H^0(X, L_{\nu'}) \rightarrow H^0(X, L_{\nu+\nu'})
\end{align*}
is surjective for all pairs of dominant weights $\nu$ and $\nu'$, 
the Cox ring of $X$ is finitely generated, namely by the $H^0(X, L_\nu)$, where 
$\nu$ is a fundamental weight. Hence, the 
subring of invariants 
\begin{align*}
\mbox{Cox}(X)^G:=\bigoplus_{\nu \in \mathcal{P}^+(\mathfrak{t} \oplus \mathfrak{t}) }
H^0(X, L_\nu)^G,
\end{align*}
where $H^0(X, L_\nu)^G$ denotes the subspace of $G$-invariant sections 
of $H^0(X, L_\nu)$, is also finitely generated. Moreover, the generators can be 
chosen to be homogeneous with respect to the $\mathcal{P}^+(\mathfrak{t} \oplus \mathfrak{t})$-grading. 
Thus, let $$s_i \in H^0(X, L_{\nu_i})^G, \, i=1,\ldots, r$$
be a set of homogeneous generators of $\mbox{Cox}(X)^G$.
For each $i \in \{1,\ldots, r\}$, let $D_i$ be a divisor with 
$\mathcal{O}_X(D_i) \cong L_{\nu_i}$. 

\begin{prop}
(i) The cone $C^G(X)$ is the convex cone in $N^1(X)_\Z$ generated by the divisors 
$D_1,\ldots, D_r$.\\

(ii) The cone $C^G(X)$ is of full dimension in $\mathfrak{t}^* \oplus \mathfrak({t}')^*$ if 
no non-zero ideal of the Lie algebra $\mathfrak{g}$ is an ideal in $\mathfrak{g}'$.
\end{prop}

\begin{proof}
We first claim that $C^G(X)$ is nonempty. Indeed, we can choose Cartan 
subalgebras $\mathfrak{t}_0 \subseteq \mathfrak{g}, \mathfrak{t}'_0 \subseteq \mathfrak{g}'$ and 
Borel subalgebras $\mathfrak{b}_0 \subseteq \mathfrak{g}, \mathfrak{b}'_0 \subseteq \mathfrak{g}'$, 
containing the respective Cartan subalgebras, such that $\mathfrak{t}_0 \subseteq \mathfrak{t}_0'$ 
and $\mathfrak{b}_0 \subseteq \mathfrak{b}_0'$. 
The restriction of linear functionals then defines a surjective linear map 
$p: (\mathfrak{t}_0)^* \rightarrow (\mathfrak{t}'_0)^*$ such that 
the dominant chamber in $(\mathfrak{t}_0')^*$, with respect to the positive system 
defined by the choice of the Borel subalgebra $\mathfrak{b}_0'$, is mapped 
into the dominant chamber of $(\mathfrak{t}_0)^*$ defined by $\mathfrak{b}'_0$. 
Since $p$ is surjective, some regular dominant weight $\lambda_0 \in (\mathfrak{t}_0')^*$ 
is mapped to a regular dominant weight $\mu_0 \in (\mathfrak{t}_0)^*$. 
If $V_{\lambda_0}$ is the highest weight representation of $G'$ - with respect to the Borel subalgebra 
$\mathfrak{b}_0'$ - with highest 
weight $\lambda_0$, and $v_{\lambda_0} \in V_{\lambda_0}$ is a highest weight vector, then 
the $G$-submodule generated by $v_{\lambda_0}$ is a highest weight representation of $G$, 
with $v_{\lambda_0}$ being a highest weight vector of highest weight 
$p(\lambda_0)=\mu_0$. Hence, the pair $(\lambda_0, \mu_0)$ defines a 
pair of regular weights in the branching cone of $(G, G')$ (defined by this new choice of 
Borel subalgebras). The corresponding line 
bundle over $X$ is thus ample and admits $G$-invariant sections. This shows that 
the cone $C^G(X)$ is nonempty.

Since $X$ is a homogeneous variety, so that every effective divisor is nef 
(cf. \cite[Example 1.4.7]{Laz}), (i) now follows by
a straightforward approximation of nef divisors by ample divisors, using the fact that 
the ample $\R$-divisors in $C^G(X)$ form a dense subset of $C^G(X)$.\\

For (ii), we refer to \cite{R10}.
\end{proof}

Let $V^G(X) \subseteq \mbox{Pic}(X) \otimes_\Z \R \cong (\mathfrak{t}^* 
\oplus (\mathfrak{t}')^*)_\R$, 
where the subscript indicates that the right hand side is viewed as a vector space over $\R$,
be the real vector space generated by the cone $C^G(X)$.

We recall (\cite{DH}) that the equivalence relation $\sim$ on the set of 
equivalence classes of line bundles over $X$ defined 
by 
\begin{align*}
L_\nu \sim L_{\nu'} \quad \mbox{if and only if} \quad  X^{ss}(L_\nu)=X^{ss}(L_{\nu'})
\end{align*} 
extends to an equivalence relation on the $G$-ample cone $C^G(X)$. 
The equivalence classes are called $GIT$-equivalence classes. 
A special case of GIT-equivalence classes is given by  the \emph{chambers}.
\begin{defin}
A GIT-equivalence class $C \subseteq C^G(X)$ is called a \emph{chamber}, if 
$X^{ss}(\ell)=X^s(\ell)$, for every $\ell \in C$.
\end{defin}
Recall here that the notions of semi-stability and stability extend to $\R$-divisors 
(cf. \cite{DH}). By VGIT (cf. \cite{DH}, \cite{R00}) there are only finitely many 
GIT-equivalence classes. 
In particular, there  exists a GIT-equivalence class $C \subseteq C^G(X)$ with 
non-empty interior. In fact, chambers exist in our setting (cf. \cite[Cor. 4.1.9]{DH}).\\
 
Let $D \in C^G(X)$ be an ample integral divisor, $L=\mathcal{O}_X(D)$ the corresponding 
line bundle, and let $Y:=X^{ss}(L)//G$, with projection morphism 
$\pi: X^{ss}(L) \to Y$, be the GIT quotient defined 
by $L$. Recall that a line bundle $L_\mu$ is said to \emph{descend to a line bundle 
on $Y$} if the sheaf on $Y$ defined by 
\begin{align*}
\mathcal{F}(U):=H^0(\pi^{-1}(U), L_\mu)^G,
\end{align*}
for $U \subseteq Y$ open, is an invertible sheaf of $\mathcal{O}_Y$-modules, i.e.,
a line bundle over $Y$, and that the isomorphism
of $G$-line bundles
\begin{align*}
\pi^*(\mathcal{F}) \cong L_\mu
\end{align*}
holds on the open subset $X^{ss}(L)$ of $X$.
In particular, there exists a $q \in \N$, such that $L^q$ descends to 
a ample line bundle on $Y$ (cf. \cite{S95}). \\

In order to study descent of line bundles from $X$ to the quotient $Y$, we recall 
some facts about the structure of $X$ as a symplectic manifold.

If $\nu \in \mathcal{P}^+(\mathfrak{t} \oplus \mathfrak{t'})$ is any regular 
dominant weight, there is a natural identification of $X$ with 
the $(U \times U')$-coadjoint orbit $\mathcal{O}_\nu \subseteq (\mathfrak{u} \oplus \mathfrak{u}')^*$, 
and this inclusion map is the moment map for the $(U \times U')$-action 
with respect to the  Kirillov-Kostant-Souriau symplectic form $\omega_\nu$ on 
$\mathcal{O}_\nu$. The moment map $\Phi_\nu: \mathcal{O}_\nu \rightarrow \mathfrak{u}^*$ 
for the action of the subgroup $U$ on $\mathcal{O_\nu}$ is simply given be restriction of 
linear functionals to $\mathfrak{u}$, i.e., we have 
\begin{align*}
(\Phi_\nu(\lambda))(x)=\lambda(x), \quad \lambda \in 
\mathcal{O}_\nu \subseteq (\mathfrak{u} \oplus \mathfrak{u}')^*, \quad 
x \in \mathfrak{u}.
\end{align*}
If $L_\nu$ is the ample line bundle defined by $\nu$, and 
\begin{align*}
f_{L_\nu}: X \rightarrow \mathbb{P}(H^0(X, L_\nu)^*) 
\end{align*} 
is the associated projective embedding, the 
moment map $\Phi_\nu$ equals the pull-back, $\Phi_{L_\nu}$,  
by $f_{L_\nu}$ of the moment map defined by the Fubini-Study symplectic form on 
the projective space $\mathbb{P}(H^0(X, L_\nu)^*)$, up to a complex scalar 
factor. Since there zero sets of moment maps will be our main interest, we will 
therefore now assume that $\Phi_\nu=\Phi_{L_\nu}$.
If $L_\nu$ and $L_{\nu'}$ are ample line bundles with 
moment maps $\Phi_{L_{\nu}}$ and $\Phi_{L_{\nu'}}$, it follows from the 
above description of these moment maps in terms of the respective coadjoint 
orbits that $\Phi_{\nu}+\Phi_{\nu'}$ is the moment map for the 
tensor product line bundle $L_{\nu}\otimes L_{\nu'}=L_{\nu+\nu'}$. That is, 
the identity 
\begin{align*}
\Phi_{L_1 \otimes L_2 }=\Phi_{L_1}+\Phi_{L_2}
\end{align*} 
holds for all ample line bundles $L_1$ and $L_2$ on $X$.

Now, if $L_\nu$ is not ample, i.e., the weight $\nu$ is not regular, 
there exists a unique parabolic subgroup $P_\nu \subseteq G \times G'$ 
and an ample line bundle $\mathcal{L}_\nu \rightarrow (G \times G')/P_\nu$ 
such that $L_\nu=q^*\mathcal{L}_\nu$, where 
$q: (G \times G')/(B \times B') \rightarrow (G \times G')/P_\nu$ is the natural 
quotient map.  
We then define $\Phi_{L_\nu}:=q^*\Phi_{\mathcal{L}_\nu}$, where 
$\Phi_{\mathcal{L}_\nu}$ is the moment map for the action of $U$ 
on $(G \times G')/P_\nu$ defined by the 
projective embedding 
\begin{align*}
f_{\mathcal{L}_\nu}: (G \times G')/P_\nu \rightarrow 
\mathbb{P}(H^0((G \times G')/P_\nu, \mathcal{L}_\nu)^*).
\end{align*}
In this way we have now defined a map $\Phi_{L_\nu}: X \rightarrow \mathfrak{u}^*$
for each effective line bundle $L_\nu$ on  $X$

\begin{lemma} \label{L: linmomentmap}
The identity 
\begin{align*}
\Phi_{L_1 \otimes L_2}=\Phi_{L_1}+\Phi_{L_2}
\end{align*}
holds for all effective line bundles $L_1$ and $L_2$ on $X$.
\end{lemma}

\begin{proof}
Let  $\lambda_1,\ldots, \lambda_r \in (\mathfrak{t} \oplus \mathfrak{t'})^*$ be 
the fundamental weights of the root system of 
$(\mathfrak{g} \oplus \mathfrak{g}', \mathfrak{t} \oplus \mathfrak{t}')$. 
It then suffices to prove that 
\begin{align*}
\Phi_{L_\nu}=m_1\Phi_{L_{\lambda_1}}+\cdots+m_r\Phi_{L_{\lambda_r}}.
\end{align*}
 for each dominant weight $\nu \in \mathcal{P}^+(\mathfrak{t} \oplus 
\mathfrak{t}')$ which has the representation 
\begin{align}
\nu=m_1\lambda_1+\cdots+m_r \lambda_r,
\label{E: sumfundw}
\end{align}
with $m_1,.\ldots, m_r \geq 0$.

Recall that, for each dominant weight $\nu \in \mathcal{P}^+(\mathfrak{t} \oplus 
\mathfrak{t}')$, the complex manifold $(G \times G')/P_\nu$ is naturally 
identified with the coadjoint orbit $\mathcal{O}_\nu \subseteq (\mathfrak{u} \oplus \mathfrak{u}')^*$. 
The Kirillov-Kostant-Souriau symplectic 
form $\omega_\nu \in \Gamma(\mathcal{O}_\nu, \bigwedge^2 T^*(\mathcal{O}_\nu))$ 
is induced from the map 
\begin{align*}
T_\nu: \mathcal{O}_\nu &\rightarrow \bigwedge^2 (\mathfrak{u}\oplus \mathfrak{u}')^*,\\
T_\nu(\lambda)(x,y)&:=\lambda([x,y]), \quad x, y \in \mathfrak{u} \oplus \mathfrak{u}'. 
\end{align*}
Now, \eqref{E: sumfundw} yields
\begin{align}
&T_\nu(\mbox{Ad}^*(u)(\nu))(x,y) \nonumber\\
&=
(m_1\mbox{Ad}^*(u)(\lambda_1)([x,y])+\cdots +(m_r\mbox{Ad}^*(u)(\lambda_r)([x,y]), \quad u \in U.
\label{E: linkir}
\end{align}
If $P_{\lambda_i}$ is the maximal parabolic subgroup of $G \times G'$ 
defined by the fundamental weight $\lambda_i$ with $m_i>0$, let $q_i: (G \times G')/P_\nu 
\rightarrow X_i:=(G \times G')/P_{\lambda_i}$ be the associated quotient map. 
It follows from \eqref{E: linkir} that 
\begin{align}
\omega_{\nu}=\sum_{i=1}^r q_i^*\omega_{\lambda_i},
\label{E: lincurv}
\end{align} 
where $\omega_{\lambda_i}$ is the Kirillov-Kostant-Souriau symplectic form on 
$(G \times G')/P_{\lambda_i}$ defined by $\lambda_i$.

Now, each $q_i$ is $(G \times G')$-equivariant, and in particular $U$-equivariant. 
Thus, if $\xi \in \mathfrak{u}$ induces the vector field 
$\xi^X$ on $X$ and the vector field $\xi^{X_i}$ on $X_i$, we have 
\begin{align*}
(dq_i)(x) (\xi^X(x))=\xi^{X_i}(q_i(x)), \quad x \in X, 
\end{align*} 
so that
\begin{align*}
d\Phi_{L_{\lambda_i}}(x)(v,\xi)&=q_i^*(d\Phi_{\mathcal{L}_{\lambda_i}})(x)(v,\xi)\\
&=d\Phi_{\mathcal{L}_{\lambda_i}}(q_i(x))(dq_i(x)v, \xi^{X_i}(x))\\
&=\omega_{\lambda_i}(dq_i(x)v, \xi^{X_i}(x))\\
&=q_i^*\omega_{\lambda_i}(v, \xi^{X}(x)), 
\quad x \in X, \, v \in T_x(X), \, \xi \in \mathfrak{u}.
\end{align*}
Hence, using \eqref{E: lincurv}, it follows that 
\begin{align*}
d\Phi_{L_\nu}(x)(v, \xi^{X}(x))&=\omega_{\nu}(v, \xi^{X}(x))\\
&=m_1d\Phi_{L_{\lambda_1}}(x)(v, \xi^{X}(x))+\cdots+
m_rd\Phi_{L_{\lambda_r}}(x)(v, \xi^{X}(x)),\\
&x \in X, v \in T_x(X), \xi \in \mathfrak{u}.
\end{align*}
Since $X$ is connected, and $U$ is semi-simple, we thus get 
\begin{align*}
\Phi_{L_\nu}=m_1\Phi_{L_{\lambda_1}}+\cdots+m_r\Phi_{L_{\lambda_r}}.
\end{align*}
This finishes the proof.
\end{proof}

We recall the following descent criterion by Kempf.

\begin{prop}(\cite[Prop. 4.2.]{KKV})
The line bundle $L_\nu$ on $X$ descends to a line bundle on $Y$
if and only if for every $x \in X^{ss}(L)$ whose $G$-orbit $G.x$ is relatively 
Zariski-closed in $X^{ss}(L)$, the stabilizer $G_x$ acts trivially on the 
fibre $(L_\nu)_x$. 
\end{prop}

We now turn to an infinitesimal version of this condition.

\noindent If the orbit $G.x$, for $x \in X^{ss}(L)$, is relatively closed in $X^{ss}(L)$, 
the stabilizer $G_x$ is the complexification of the stabilizer $U_x$, i.e., 
$G_x=(U_x)^\C$ (cf. \cite[Prop. 1.6, Prop. 2.4]{S95}).
 
Let $D_{L_\nu}$ be the Chern connection of $L_\nu$ 
with respect to the unique Hermitian metric $h_{L_\nu}$ defined 
by the 2-form $\omega_\nu$. 
(Note that this is defined by pulling back the natural Hermitian metric 
from an ample line bundle on the flag variety $(G \times G')/P_\nu$.)

For a tangent vector $v \in T_\eta(L_\nu), \,\eta \in L_\nu$, 
let $v_h$ be the horizontal component of $v$ as defined by the 
splitting of the tangent bundle $T(L_\nu)$ induced by the 
connection $D_{L_\nu}$.  Since the group $U$ acts on the line bundle 
$L_\nu$ as automorphisms preserving the Hermitian metric $h_{L_\nu}$, 
for any $\xi \in \mathfrak{u}$, the vector field $\xi^{L_\nu}$ on $L_\nu$ defined 
by $\xi$ is given by
\begin{align*}
\xi^{L_\nu}(\eta)=(\xi^{L_\nu}(\eta))_h+\Phi_{L_\nu}(x)(\xi)\zeta(\eta),
\quad \eta \in (L_\nu)_x,
\end{align*}
where $\zeta$ is the vector field on $L_\nu$ generating the 
$S^1$-action on $L_\nu$ defined by fibrewise multiplication 
(cf. \cite[Theorem 3.3.1.]{Kos}).

In particular, the infinitesimal action of the Lie algebra
$\mathfrak{u}_x$ of the stabilizer $U_x$ of $x$ on the fibre $(L_\nu)_x$ is given by 
\begin{align*}
\xi^{L_\nu}(\eta)=\Phi_{L_\nu}(x)(\xi)\zeta(\eta), 
\quad \eta \in (L_\nu)_x.
\end{align*}
A necessary condition for the descent of the line bundle 
$L_\nu$ to a line bundle on $Y$ is thus that the 
condition 
\begin{align}
\Phi_{L_\nu}(x)(\xi)=0
\label{E: infdescent}
\end{align}
hold for every $x \in X^{ss}(L)$ for which the orbit $G.x$ is relatively 
closed in $X^{ss}(L)$, and $\xi \in \mathfrak{u}_x$.
On the other hand, if the condition \eqref{E: infdescent} holds, 
the identity component $(G_x)^0$ acts trivially 
on $(L_\nu)_x$ for every $x \in X^{ss}(L)$ with relatively closed $G$-orbit. 
Therefore the action of $G_x$ on $(L_\nu)_x$ factorizes through an 
action of the finite group $G_x/(G_x)^0$ of connected components of $G_x$.
If $q=q(x)$ is the order of this group, the action of $G_x$ on the fibre 
$(L_\nu^q)_x$ at $x$ of the $q$-th tensor power of $L_\nu$ is 
thus trivial. Since there are only finitely many conjugacy classes 
of stabilizers of points $x \in X^{ss}(L)$ with relatively closed 
$G$-orbit $G.x \subseteq X^{ss}(L)$, the condition \eqref{E: infdescent} implies the 
existence of a uniform $q \in \N$ such that $G_x$ acts 
trivially on $(L_\nu^q)_x$, for all $x \in X^{ss}(L)$ with relatively 
closed $G$-orbit. Hence, the line bundle $L_\nu^q$ descends to 
$Y$. We have thus proved the following proposition.

\begin{prop} \label{P: infdescent}
For a line bundle $L_\nu$ on $X$, the following are equivalent:\\

\noindent (i) There exists a natural number $q$, such that $L_\nu^q$ descends to 
a line bundle on $Y$,\\

\noindent (ii) For every point $x \in X^{ss}(L)$ for which the $G$-orbit $G.x$ 
is relatively Zariski-closed in $X^{ss}(L)$, and $\xi \in \mathfrak{u}_x$, the 
condition
\begin{align*}
\Phi_{L_\nu}(x)(\xi)=0
\end{align*}
holds.
\end{prop}

We are now ready to state the main theorem of this section.

\begin{thm} \label{T: descent}
\noindent (i) Assume that the subgroup $G$ is semi-simple. 
Let $F$ be a cell in $C^G(X)$, and let $E_1,\ldots, E_m \in C^G(X)$ be divisors 
such that for some nonempty open subset $V \subseteq (\R^+)^m$, the 
divisors $t_1E_1+\cdots+t_mE_m, \, (t_1,\ldots, t_m) \in V$, lie in the interior of $F$.
Then, for every divisor $E$ in the subgroup of $\mbox{Pic}(X)$ generated by the 
divisors $E_1,\ldots, E_m$, there exists an $m \in \N$ such that 
the line bundle $\mathcal{O}_X(mE)$ descends to a line bundle 
on $Y=X^{ss}(F)//G$. In particular, this holds if the closure $\overline{F}$ is 
a rational polyhedral cone generated by the divisors $E_1,\ldots, E_m \in C^G(X)$\\

\noindent (ii) If $F=C$ is a chamber, and $G \subseteq G'$ is a reductive 
subgroup, then each integral divisor $E$ on $X$ admits a multiple $mE$, for some  
$m\in \N$, such that $\mathcal{O}_X(mE)$ descends to a line bundle on $Y$
\end{thm}

\begin{proof}
For the first part it suffices to prove that the claim holds for all the generators 
$E_i, \, i=1,\ldots, m$. 

Let $D$ be a divisor  in the interior of $F$ which can be written as a 
linear combination $D=\sum_{j=1}^m t_j E_j$, for some $(t_1,\ldots, t_m) \in V$. 
Then, for any $i \in \{1,\ldots, m\}$, 
the divisor $D$ can be expressed as a  
linear combination 
\begin{align*}
D=\sum_{j=1}^m t_j E_j, \quad t_1,\ldots, t_m \geq 0,
\end{align*}
where $t_i>0$. By renumbering the divisors $E_i$, if necessary, me 
may assume that $i=1$, and that the linear combination is 
of the form 
\begin{align*}
D=\sum_{j=1}^\ell t_jE_j, \quad t_1,\ldots, t_\ell >0,
\end{align*} 
where $2 \leq \ell \leq r$.
Let $\Phi_j: X \rightarrow \mathfrak{u}^*$  be the moment map for the action of 
$U$ on $X$ with respect to the line bundle $\mathcal{O}_X(E_j), \, j=1,\ldots, m$, 
and let $\Phi_D: X \rightarrow \mathfrak{u}^*$ be the 
moment map for the $U$-action with respect to $L=\mathcal{O}_X(D)$. 
Then, by Lemma \ref{L: linmomentmap},
\begin{align*}
\Phi_D=\sum_{j=1}^\ell t_j \Phi_j.
\end{align*}
Now, let $x \in X^{ss}(L)$ be a point for which the orbit $G.x \subseteq X^{ss}(L)$ 
is closed in the relative Zariski topology, and let $\xi \in \mathfrak{u}_x$.  
Since the line bundle $\mathcal{O}_X(mD)$ descends, we have
\begin{align*}
\Phi_D(x)(\xi)=0.
\end{align*}
We now prove that $\Phi_1(x)(\xi)=0$ also holds. Indeed, for any tuple 
$(\tau_1, \ldots, \tau_\ell) \in \R^\ell$ of 
positive rational numbers for which $D_\tau:=\tau_1E_1+\cdots+\tau_\ell E_\ell \in F$, 
the $\Q$-divisor $D_\tau$ admits some integral multiple $pD_\tau$ such 
that $X^{ss}(\mathcal{O}_X(pD_\tau))=X^{ss}(L)$. Therefore  
we have $\tau_1\Phi_1(x)(\xi)+\cdots+\tau_\ell\Phi_\ell(x)(\xi)=0$, by Proposition 
\ref{P: infdescent}, since some 
positive power of the line bundle $\mathcal{O}_X(pD_\tau)$ descends 
to a line bundle on $Y$. Hence, the set 
\begin{align*}
\{(\tau_1,\ldots, \tau_\ell) \in \R^\ell \mid \tau_1 \Phi_1(x)(\xi)+\cdots +
\tau_\ell \Phi_\ell(x)(\xi)=0 \}
\end{align*}
contains the open neighbourhood $V$ of the point $(t_1,\ldots, t_\ell)$.
It follows that the map 
\begin{align*}
f: \R^\ell \rightarrow \R, \quad f(\tau_1,\ldots, \tau_\ell):=
\tau_1\Phi_1(x)(\xi)+\cdots +\Phi_\ell(x)(\xi)
\end{align*}
is the zero map. In particular, $f(1,0,\ldots, 0)=\Phi_1(x)(\xi)=0$.
The Lie algebra $\mathfrak{u}_x$ of the stabilizer $U_x$ thus acts 
trivially on the fibre of $\mathcal{O}_X(E_1)$ at $x$. By Proposition 
\ref{P: infdescent} there thus 
exists a $q \in \N$ such that the line bundle  
$\mathcal{O}_X(qE_1)$ descends to a line bundle on $Y$. 
This proves the first claim.

If $G$ is semisimple, then, for divisors $E \in C^G(X)$ (ii) follows from (i) using 
$$(E_1,\ldots, E_m)=(D_1,\ldots, D_r).$$  Hence, every integral divisor in $V^G(X)$ admits 
an integral multiple which descends.

On the other hand, for a general reductive $G \subseteq G'$, every stabilizer $G_x$ of a point 
$x \in X^{ss}(C)$ is finite since each $C$-semi-stable point is $C$-stable. 
Hence, if $E$ is an arbitrary integral divisor on $X$, and $m$ is the smallest common multiple 
of the orders of the finite groups $G_x$ above, the line bundle $\mathcal{O}_X(mE)$ descends to $Y$, 
for each. This proves the second claim. 
\end{proof}

From now on we will focus on GIT quotients $Y=X^{ss}(D)//G=X^{ss}(C)//G$, where 
$D \in C^G(X)$ is an ample divisor in a chamber $C$. 

Since the cone $C^G(X)$, where 
divisors are naturally defined with integral weights, is of full dimension in 
the vector space $V^G(X)$, the subset 
$\{D_1,\ldots, D_r\}$ contains an $\R$-basis for $V^G(X)$. Without loss of generality, 
assume that $\{D_1,\ldots, D_s\}$ is such a basis.
For each $D_i, \, i=1,\ldots, s$, let $m_i$ be the smallest natural number such 
that $\mathcal{O}_X(m_iD_i)$ descends to a line bundle $\mathscr{L}_i$ on $Y$. 
Put $F_i:=m_iD_i, \, i=1,\ldots, r$, and define the $\R$-linear map
\begin{align}
\sigma: V^G(X) & \rightarrow \mbox{Pic}(Y) \otimes_\Z \R,
 \label{E: linmapofcones}\\
\sigma(x_1F_1+\cdots+x_s F_s)&:=x_1 \mathscr{L}_1+\cdots + x_s\mathscr{L}_s, 
\quad x_1,\ldots, x_s \in \R. \nonumber
\end{align}

%{\bf The discussion below of the base loci of $G$-invariant sections has to be 
%postponed until later.}

%If $L \rightarrow X$ is a line bundle, let 
%\begin{align*}
%B^G(X):=\{x \in X \mid \forall m \in \N \forall s \in H^0(X, L^m)^G s(x)=0\}
%\end{align*}
% be the common zero set of all $G$-invariant sections of all positive powers 
%of $L$. The set $B^G(X)$ is then a $G$-invariant and Zariski-closed subset of $X$. 
%Let $E_1,\ldots, E_\ell$ be the irreducible components of $B^G(L)$ of 
%codimension one, and, for $j=1,\ldots, \ell$, let $\mathcal{O}_X(E_i)$ be 
%the corresponding line bundle. Let $s_{E_j} \in H^0(X, \mathcal{O}_X(E_j))$ be 
%the defining section of $E_j$.

\section{A Mori dream space quotient}

In this section we refine the choice of the quotient $Y$ from the previous section.

\begin{prop}
There exists an ample divisor $D \in C^G(X)$ such that the set of unstable points $X \setminus X^{ss}(D)$ is 
of codimension at least two. Moreover, $D$ can be chosen to lie in a chamber.
\end{prop} 

\begin{proof}
Let $D' \in C^G(X)$ be an arbitrary divisor, and let  $Y=X^{ss}(D')//G$ be the 
associated quotient.  By replacing  $L'=\mathcal{O}_X(D')$ with some power, if necessary, we may 
assume that $L'$ descends to a line bundle $L'_0$ on $Y$. 

Since $L_0'$ is semi-ample, the multiplication maps
\begin{align*}
H^0(Y, (L'_0)^k) \otimes H^0(Y, (L'_0)^\ell) \rightarrow H^0(Y, (L_0')^{k+\ell})
\end{align*}
are surjective for $k$ and $\ell$ sufficiently big (\cite[Example 2.1.29]{Laz}). In particular, for 
$k$ big enough, the section ring $\bigoplus_{m=1}^\infty H^0(Y, (L'_0)^{km})$ is generated in 
degree one. Hence, the ring of invariants $\bigoplus_{m=1}^\infty H^0(X, (L')^{km})^G$ is generated in 
degree one  (cf. Remark \ref{R: riemanninv}). 
By again replacing $L'$ by a sufficiently high power, we may assume that $k=1$, so that 
the linear series $\bigoplus_{m=1}^\infty H^0(X, (L')^{m})^G$ is generated in degree one. 
The set of unstable points $X \setminus X^{ss}(D')$ thus equals the common zero 
set of all sections in $H^0(X, L')^G$.
Let $E_1,\ldots, E_\ell$ be the one-dimensional irreducible components of $X \setminus X^{ss}(D')$, and, 
for each $i$, let $m_i \in \N$ be the smallest order to which some section in $H^0(X, L'_0)^G$ 
vanishes along $E_i$. Since $X \setminus X^{ss}(D)$ is $G$-invariant,  each $E_i$ is also $G$-invariant. 
Hence, each $E_i$ is the zero set of a section $\xi_i \in H^0(X, \mathcal{O}_X(E_i))^G$. 
Now, for each $g \in G$, the section $g.\xi_i \in H^0(X, \mathcal{O}_X(E_i))$ vanishes on the zero 
set, $E_i$, of $\xi_i$, so that, by the normality of $X$, $g.\xi_i/\xi_i$ defines a global regular 
function on $X$, i.e., $g.\xi_i=c(g)\xi_i$, for some $c(g) \in \C$. The function $g \mapsto c(g)$ 
then defines a character of $G$, so it has to be constant by the semisimplicity of $G$; that is, 
the section $\xi_i$ is $G$-invariant.

Since each section $s \in H^0(X, L')^G$ is divisible by the section 
$$\xi:=\xi_1^{m_1} \cdots \xi_\ell^{m_\ell} \in H^0(X, \mathcal{O}_X(E))^G,$$ where 
$E:=m_1E_1+\cdots +m_\ell E_\ell$, 
the map
\begin{align*}
\varphi:& \bigoplus_{k=1}^\infty H^0(X, \mathcal{O}_X(k(D'-E)))^G \rightarrow 
\bigoplus_{k=1}^\infty H^0(X,  \mathcal{O}_X(kD'))^G\\
\varphi(s)&:=s \cdot \xi^k, \quad s \in H^0(X, \mathcal{O}_X(k(D'-E)))^G
\end{align*}
defines an isomorphism of graded rings. Put $D'':=D'-E$ and $L'':=\mathcal{O}_X(D'')$. 
By construction, the unstable locus of $\mathcal{O}_X(D'')$ is a subset of the 
union of the irreducible components of $X \setminus X^{ss}(D')$  which are of 
codimension at least two.
We now claim that the  divisor $D''$ can be chosen to lie in the interior of $C^G(X)$. 
Indeed, if we start with a divisor $D_1'$, and the resulting divisor from the above 
construction, $D''_1$, would happen to lie in some face $F_1$ (of maximal dimension) of the boundary 
of $C^G(X)$, we proceed as 
follows. Let $F_2$ be an extremal ray of $C^G(X)$, i.e., a face of minimal dimension, 
such that $F_2 \cap F_1=\emptyset$ (we may assume here that $\mbox{dim}\, C^G(X)>1$, since $C^G(X)$ 
will otherwise lie in the ample cone of $X$),  and let $D_2' \in F_2$ be 
an integral divisor in this face. By applying the above construction to $D_2'$, we get 
a decomposition of $D_2'$ as the sum $D_2'=D_2''+E_2$, where $D_2''$ has an unstable locus 
of codimension at least two. Since $D_2'$ lies in the face $F_2$, we must also have 
$D_2'' \in F_2$ and $E_2 \in F_2$.   Then, since $F_1$ and $F_2$ are disjoint faces of 
$C^G(X)$,  the divisor $D'':=D_1''+D_2''$ lies in the interior of $C^G(X)$. In particular, $D''$ 
is ample. 

In order to see that $D''$ also has an unstable locus of codimension at least two, 
assume that $X \setminus X^{ss}(D'')$ contains some divisor $Z$.  Then, by 
replacing $D_1''$ by some multiple $mD_1'', m \in \N$, we can, using 
the fact that the unstable locus of $D_1''$ contains no divisor,   
assume that the exists a section $s_1 \in H^0(X, \mathcal{O}_X(D_1''))^G$ which does not vanish on 
$Z$. However, for every $m\in \N$, and every section $t \in H^0(X, \mathcal{O}_X(mD_2''))^G$, 
the product section $s_1^m \cdot t \in H^0(X, \mathcal{O}_X(mD''))^G$  
vanishes on $Z$. It follows that the divisor $Z$ lies in the stable base locus 
of $D_2''$; a contradiction.  This shows that $X \setminus X^{ss}(D'')$ is of codimension at least 
two.
 
If $D''$ happens to lie in a chamber of $C^G(X)$, we are done. Otherwise, $D''$ lies in a 
cell $F$. Now choose an effective divisor $E$ and a natural number $k$ so that the divisor $kD''-E$ 
lies in some chamber $C$. We now claim that the set of unstable points of $C$, i.e., of 
$kD''-E$, is of codimension at least two. Indeed, the 
section ring  $$R(kD''-E, E)=\bigoplus_{m_1,m_2 \in \Z} H^0(X, \mathcal{O}_X(m_1(kD''-E)+m_2E))$$ is 
finitely generated since $kD''-E$ and $E$ are semi-ample divisors (\cite[Lemma 2.8]{HK}). 
Hence, the invariant ring 
$$R(kD''-E, E)^G=\bigoplus_{m_1, m_2 \in \Z}H^0(X, \mathcal{O}_X(m_1(kD''-E)+m_2E))^G$$ 
is also finitely generated. Thus, there exists an $\ell \in \N$ such that, for $m \geq \ell$, 
any invariant section $s \in H^0(X, \mathcal{O}_X(mkD'')^G$ can be 
written as 
\begin{align*}
s=\sum_{j=1}^p s_j t_j, 
\end{align*} 
for some $s_j \in H^0(X, \mathcal{O}_X(m(kD''-E)))^G, \, t_j \in H^0(X, \mathcal{O}_X(mE))^G, j=1,\ldots, p$.
Now, if all invariant sections $t \in H^0(X, \mathcal{O}_X(m(kD''-E)))^G, \, m \in \N$, would vanish 
on some divisor $Z$ of $X$, then every $s \in H^0(X, \mathcal{O}_X(mkD''))^G$, 
for $m$ sufficiently big, would also vanish  on the divisor $Z$ - in contradiction to the fact 
that the unstable locus of $D''$ contains no divisors. Hence, the set of unstable points of 
$D:=kD''-E$ is of codimension at least two.
\end{proof}

Now, let $D \in C^GX)$ be an ample divisor which satisfies the conditions in the above proposition, 
and put $L:=\mathcal{O}_X(D)$.

\begin{thm} \label{T: Pic(Y)}

(i) For any integral divisor $F \in C^G(X)$ such that 
$\mathcal{O}_X(F)$ descends to a line bundle $\mathcal{O}_Y(F_Y)$ 
on $Y$, 
\begin{align*}
H^0(Y, \mathcal{O}_Y(F_Y)) \stackrel{\pi^*}{\cong} 
H^0(X^{ss}(L), \mathcal{O}_X(F)\mid_{X^{ss}(L)})^G
\cong H^0(X, \mathcal{O}_X(F))^G.
\end{align*}
\\ 

\noindent (ii) The Picard group $\mbox{Pic}(Y)$ is finitely generated, and 
the restriction of the map \eqref{E: linmapofcones} defines an isomorphism 
of cones
\begin{align*}
\sigma\mid_{C^G(X)}: C^G(X) \stackrel{\cong}{\rightarrow} \overline{\mbox{Eff}}(Y) \subseteq 
\mbox{Pic}(Y) \otimes_\Z \R.
\end{align*}

\end{thm}

\begin{proof}
First of all, the natural restriction map $\mbox{Pic}(X) \rightarrow \mbox{Pic}(X^{ss}(L))$ 
defines an isomorphism of groups, since $X \setminus X^{ss}(L)$ does not 
contain any divisors. 

For (i), we first note that the first isomorphism holds by the definition of the 
sheaf $\mathcal{O}_Y(F_Y)$. Moreover, since $X$ is normal, and 
the unstable locus of $L$ is of codimension at least two, any section 
$s \in H^0(X^{ss}(L), \mathcal{O}_X(F) \mid_{X^{ss}(L)})$ extends 
uniquely to a section $S$ of $\mathcal{O}_X(F)$ over $X$. If $s$ is 
$G$-invariant, then so is $S$, since the identity $g.S=S$ holds 
on $X$ if it holds on $X^{ss}(L)$.
This shows that the second isomorphism holds. This proves (i).

The natural restriction map 
\begin{align}
\mbox{Pic}(X) \rightarrow \mbox{Pic}(X^{ss}(L))
\label{E: restrpic}
\end{align}
defines an isomorphism of groups, since $X \setminus X^{ss}(L)$ does not 
contain any divisors. 

If $E$ is a divisor on $Y$, the divisor $\pi^*(E)$ on $X^{ss}(L)$ extends to 
divisor on $X$, which we will also denote by $\pi^*(E)$, by taking its 
closure in $X$. The map $\pi^*$ extends to an injective linear 
map $\pi^*: \mbox{Pic}(Y) \otimes_\Z \R \rightarrow \mbox{Pic}(X) \otimes_\Z \R$
of real vector spaces. Since \eqref{E: restrpic} is an isomorphism of 
groups, the map $$\sigma: V^G(X) 
\rightarrow \mbox{Pic}(Y) \otimes_\Z \R$$ is an inverse to $\pi^*$, so that we 
have an isomorphism
$V^G(X) \stackrel{\cong}{\rightarrow} \mbox{Pic}(Y) \otimes_\Z \R$
of real vector spaces. The map $\pi^*$ maps $\mbox{Pic}(Y)$ injectively 
into the finitely generated abelian group $\mbox{Pic}(X)$, so that $\mbox{Pic}(Y)$ is 
also finitely generated. Hence, $N^1(Y)_\R \cong \mbox{Pic}(Y) \otimes_\Z \R$, so that 
$\sigma$ indeed defines an isomorphism between $V^G(X)$ and $N^1(Y)_\R$.
By (i), we also have $\sigma(C^G(X))=\overline{\mbox{Eff}}(Y)$. 
\end{proof}

\begin{rem} \label{R: riemanninv}
The second isomorphism in part (i) of the above theorem actually also holds
when the unstable locus $X \setminus X^{ss}(L)$ is not of codimension at least two. 
Indeed, we can argue as in the proof of \cite[Theorem 2.18]{S95}.

Recall that the set of semi-stable points $X^{ss}(L)$ is the set of 
points $x \in X$ for which the closure $\overline{G.x}$ of the $G$-orbit of 
$x$ intersects the zero set $\Phi_D^{-1}(0)$ of the moment map $\Phi_D$ 
defined by $D$. Hence, any $G$-invariant section 
$s \in H^0(X^{ss}(L), \mathcal{O}_X(F)\mid_{X^{ss}(L)})$ is 
uniquely determined by its values on the zero set $\Phi_D^{-1}(0) \subseteq X^{ss}(L)$.
For any $U$-invariant Hermitian metric $h$ on $\mathcal{O}_X(F)$, the 
function $x \mapsto h(s(x), s(x))$ is therefore bounded on $X^{ss}(L)$. 
By the Riemann Extension Theorem, $s$ thus admits a unique extension to a section 
$S \in H^0(X, \mathcal{O}_X(F))$. 
For any  $g \in G$, the identity $gS=S$ holds since it holds over the 
open set $X^{ss}(L)$, by the $G$-invariance of $s$. Hence, 
$S \in H^0(X, \mathcal{O}_X(F))^G$. This proves the 
second isomorphism. 
\end{rem}

\begin{rem} \label{R: coneid}
Even though the branching cone is usually defined in the more general setting of a pair $(G, G')$, 
where $G$ and $G'$ are merely reductive, the second part of the above theorem can not 
hold in this generality. Indeed, if both $G'$ and $G$ have nontrivial centres, and 
$\lambda' \in (\mathfrak{t}')^*$ is an integral weight which exponentiates to a nontrivial 
character of $G'$, i.e., of the centre of $G'$, and $\lambda \in (\mathfrak{t})^*$ is the 
restriction of $\lambda$, and thus exponentiates to a character of $G$, then 
both $(\lambda, \lambda')$ and $(-\lambda, -\lambda')$ lie in the branching cone 
of $(G, G')$, i.e., in $C^G(X)$. If the restricted character of $G$ is nontrivial, 
the points $(\lambda, \lambda')$ and $(-\lambda, -\lambda')$ in $C^G(X)$ are distinct; they 
both represent the trivial line bundle over $X$, but equipped with two distinct structures of 
a $G$-line bundle. In particular, the cone $C^G(X)$ then contains the real line 
through the point $(\lambda, \lambda')$. However, the pseudo-effective cone of a 
projective variety can not contain any lines (\cite[Lemma 4.6]{LM09}).
\end{rem}

\begin{thm} \label{T: MDS}
The quotient $Y:=X^{ss}(D)//G$ is a Mori dream space.
\end{thm}

\begin{proof}
The variety $Y$, being the GIT quotient of a normal variety, is normal. 
Moreover, since $Y$ is a geometric quotient of the $\Q$-factorial variety $X^{s}(D)$, 
it is $\Q$-factorial (\cite[Lemma 2.1]{HK}). By Theorem \ref{T: Pic(Y)}, the 
group $\mbox{Pic}(Y)$ is finitely generated. In order to prove the claim, 
it thus suffices to prove the Cox ring of $Y$ is finitely generated.

For this, we first choose a rational polyhedral cone 
$\mathcal{C} \subseteq \mbox{Pic}(Y) \otimes \R=N^1(Y)_\R$ such that $\mathcal{C}$ is 
generated by an integral basis $\{Z_1,\ldots, Z_s\}$ for the $\R$-vector space $N^1(Y)_\R$, and 
$\overline{\mbox{Eff}}(Y)$ lies in the interior of $\mathcal{C}$. (This is possible, since 
the pseudo-effective cone of a projective variety does not contain any lines.) 
Then, the divisors $Z_i$ pull back to divisors $\pi^*Z_i$ on $X$ with 
\begin{align*}
H^0(X, \mathcal{O}_X(\pi^*(m_1Z_1+\cdots +m_sZ_s)))^G &\cong 
H^0(Y, \mathcal{O}_Y(m_1Z_1+\cdots +m_sZ_s)), \\ 
& m_1,\ldots, m_s \in \Z.
\end{align*}
Now, $H^0(X, \mathcal{O}_X(\pi^*(m_1Z_1+\cdots +m_s Z_s)))^G \neq \{0\}$ if and only if 
$$\pi^*(m_1Z_1+\cdots+m_sZ_s) \in C^G(X).$$ Let $\Gamma \subseteq V^G(X)$ be the 
integral lattice generated by the divisors $\pi^*Z_i, \,i=1,\ldots, s$, and put 
$S(\Gamma):=\Gamma \cap C^G(X)$. By the fact that each integral divisor in
$V^G(X)$ admits a multiple which descends to $Y$, the cone $C^G(X)$ is the 
closed convex cone generated by the semigroup 
$S(\Gamma)$.  Since $C^G(X)$ is a rational polyhedral cone,  Gordan's 
lemma shows that the semigroup $S(\Gamma)$ 
is finitely generated, say, by divisors $E_1,\ldots, E_\ell$. 
Moreover, since every effective divisor on $X$ is semi-ample, the section ring 
\begin{align*}
R(E_1,\ldots, E_\ell):=\bigoplus_{m_1,\ldots, m_\ell \in\N_0} 
H^0(X, \mathcal{O}_X(m_1E_1+\cdots+m_\ell E_\ell))
\end{align*} 
is finitely generated, so that 
\begin{align*}
\bigoplus_{m_1,\ldots, m_s \in \Z} H^0(X, \mathcal{O}_X(\pi^*(m_1Z_1+\cdots +m_sZ_s)) 
\cong R(E_1,\ldots, E_\ell)
\end{align*}
is a finitely generated ring. By taking invariants, it follows that 
\begin{align*}
\mbox{Cox(Y)}=\bigoplus_{m_1,\ldots, m_s \in \Z} 
H^0(X, \mathcal{O}_X(\pi^*(m_1Z_1+\cdots +m_sZ_s))^G 
\end{align*}
is a finitely generated ring.
\end{proof}

%\begin{rem}
%The cone $C^{\mbox{eff}}(Y)$ is not by definition the pseudo-effective 
%cone $\overline{\mbox{Eff}}(Y)$ of $Y$, since the latter is a cone in the 
%quotient $N^1(Y)_\R$ of $\mbox{Pic}(Y)_\R$. It is not clear to us whether 
%the map $\pi^*$ actually induces an 
%injection of $N^1(Y)_\R$ into $N^1(X)_\R$. Indeed, if the complement of $X^{ss}(L)$, 
%i.e., the stable base locus of the linear system $\bigoplus_{k \in \N}H^0(X, L^k)^G$,
%is of codimension one, it is not clear that all intersection products of 
%pull-back divisors $\pi^*E$ with curves in $X$ can be taken with 
%curves in $X^{ss}(L)$. It seems, however, reasonable to expect that we have an 
%identification of $C^{\mbox{eff}}(Y)$ with $\overline{\mbox{Eff}}(Y)$ 
%if $X \setminus X^{ss}(L)$ is of codimension at least two.
%\end{rem}

\section{Global branching laws}

In this section we use the identification of the branching cone for the pair $(G, G')$
with the pseudo-effective cone $\overline{\mbox{Eff}}(Y)$ to study the global 
branching laws.

We first briefly recall the construction of Okounkov bodies. For a thorough 
treatment, we refer to the seminal papers \cite{Ok96}, \cite{KK09}, and \cite{LM09}.  
Let $n:=\mbox{dim} Y$. 
An \emph{admissible flag} of subvarieties of $Y$ is a flag
\begin{align*}
\{p\}=Y_n \subseteq Y_{n-1} \subseteq \cdots \subseteq Y_1 \subseteq Y_0:=Y 
\end{align*}
of normal irreducible subvarieties, where $Y_i$ has codimension $i$, $i=0,\ldots, n$, 
and where the point $p$ is a non-singular point of each $Y_i$.
Let 
\begin{align*}
v: \C(Y)^\times \rightarrow \Z^n
\end{align*}
be the valuation on the ring of rational functions defined by the flag $Y_\bullet$, 
where $\Z^n$ is equipped with the lexicographic order.
 
If $E$ is an effective divisor on $Y$, by identifying the section ring 
$$R(E)=\bigoplus_{k=0}^\infty H^0(Y, \mathcal{O}_Y(kE))$$ of $E$ 
with a subring of $\C(Y)$, the valuation $v$ yields a \emph{valuation-like function}
\begin{align*}
\bigsqcup_{k \geq 0} H^0(Y, \mathcal{O}_Y(kE)) \setminus \{0\} \rightarrow \N_0^n,
\end{align*}  
i.e., a function having the ring-theoretic properties of a valuation, although it 
is only defined on nonzero homogeneous elements.

Now, let $\Sigma \subseteq \overline{\mbox{Eff}}(Y)$ be the semigroup generated by 
all effective divisors. Using values of all effective divisors in $\Sigma$, we 
define the semigroup
\begin{align*}
S_{Y_\bullet}(Y):=\{(v(s), E) \in \N_0^n \times \Sigma \mid s \in H^0(Y, \mathcal{O}_Y(E)) 
\setminus \{0\}\}.
\end{align*}
Finally, we define the \emph{global Okounkov body} for $Y$, with respect to 
the flag $Y_\bullet$, to be the closed convex cone 
\begin{align*}
\Delta_{Y_\bullet}(Y) \subseteq \R^n \times \overline{\mbox{Eff}}(Y)
\end{align*}
generated by the semigroup $S_{Y_\bullet}(Y)$.

For a fixed effective divisor $E$, we define the semigroup
\begin{align*}
S_{Y_\bullet}(E):=\{(v(s), k) \in \N_0^n \times \N_0 \mid s \in H^0(Y, \mathcal{O}_Y(kE)) 
\setminus \{0\}\}, 
\end{align*}
and the Okounkov body for $E$, with respect to the flag $Y_\bullet$, as 
the closed convex hull
\begin{align*}
\Delta_{Y_\bullet}:=\overline{\mbox{conv} 
\left\{\frac{v(s)}{k} \mid (v(s), k) \in S_{Y_\bullet}(E)\right\}} \subseteq \R^n.
\end{align*} 

If $p_2: \R^n \times \overline{\mbox{Eff}}(Y) \rightarrow \overline{\mbox{Eff}}(Y)$ denotes 
the projection onto the second factor, the identity
\begin{align}
p_2^{-1}(E) \cap \Delta_{Y_\bullet}(Y)=\Delta_{Y_\bullet}(E)
\label{E: okslice}
\end{align}
then holds for every effective divisor $E \in \Sigma$ (cf. \cite{LM09}). It therefore makes sense
to define Okounkov bodies for $\R$-divisors by
\begin{align}
\Delta_{Y_\bullet}(\xi):=p_2^{-1}(\xi) \cap \Delta_{Y_\bullet}(Y), \quad \xi \in \overline{\mbox{Eff}}(Y).
\label{E: okrdiv}
\end{align}

For each $E \in \Sigma$, let $d(E) \in \{0,\ldots, n\}$ be the dimension of 
$\Delta_{Y_\bullet}(E)$, i.e., the dimension of the smallest subspace of 
$\R^n$ containing the convex compact set $\Delta_{Y_\bullet}(E)$, and 
let $\mbox{Vol}_{d(E)}(\Delta_{Y_\bullet}(E))$ denote the volume of $\Delta_{Y_\bullet}(E)$ 
with respect to the Lebesgue measure on $\R^{d(E)}$. In particular, if 
$E$ is big, that is, in the interior of $\overline{\mbox{Eff}}(Y)$, then $d(E)=n$.
The volume of the Okounkov body $\Delta_{Y_\bullet}(E)$, for a big divisor $E \in \Sigma$, 
encodes the asymptotics of the spaces of sections $H^0(Y, \mathcal{O}_Y(E)), \, k \in \N$, 
by the identity
\begin{align}
\lim_{k \rightarrow \infty}\frac{\mbox{dim}\, H^0(Y, \mathcal{O}_Y(kE))}{k^n}
=\mbox{Vol}_n(\Delta_{Y_\bullet}(E))
\label{E: volid}
\end{align}
(cf. \cite{LM09}). Hence, we have the following theorem. 

\begin{thm}
If $F \in C^G(X)$ is an ample divisor in the 
interior of the ample cone $C^G(X)$ which descends to a divisor $F_Y$ on $Y$,  
then 
\begin{align*}
\lim_{k \rightarrow \infty}\frac{\mbox{dim}\, H^0(X, \mathcal{O}_X(kF))^G}{k^n}
=\mbox{Vol}_n(p_2^{-1}(F_Y) \cap \Delta_{Y_\bullet}(Y)).
\end{align*}
\end{thm}

\begin{proof}
The claim follows immediately from the identities \eqref{E: okslice}, \eqref{E: volid}, and 
the isomorphism 
\begin{align*}
H^0(X, \mathcal{O}_Y(kF))^G \cong H^0(Y, \mathcal{O}_Y(kF_Y)), \quad k \in \N.
\end{align*}
\end{proof}

As an immediate corollary of the log concavity of the volume function 
in the interior of the pseudo-effective cone (\cite[p. 157]{Laz}, \cite[Cor. 4.12]{LM09}), 
we obtain the following generalization of Okounkov's result (\cite{Ok96}) on log concavity 
of asymptotic multiplicities.

\begin{thm}
The function 
\begin{align*}
f: \mbox{Int}\,(\overline{\mbox{Eff}}(Y)) \rightarrow \R, \quad f(\xi):=\log (\mbox{Vol}_n(\Delta_{Y_\bullet}(\xi)))
\end{align*}
(cf. \eqref{E: okrdiv}) is concave.
\end{thm}

\begin{rem} \label{R: shapeok}
So far in this section we have been working with an arbitrary admissible flag $Y_\bullet$, 
and the shape of the global Okounkov body $\Delta_{Y_\bullet}(Y)$ of course depends on this flag. 
An interesting question is now whether $Y$ admits an admissible flag $Y_\bullet$ for which 
$\Delta_{Y_\bullet}(Y)$ is a rational polyhedral cone. In particular, each single 
Okounkov body $\Delta_{Y_\bullet}(E)$, being a fibre of $\Delta_{Y_\bullet}(Y)$,  would then be a rational 
polyhedral polytope, and the asymptotics of the corresponding branching law would be (approximately at least)
given by counting integral points of this polytope. 

Note that we can at least make one slice of $\Delta_{Y_\bullet}(Y)$ rational polyhedral 
by a suitable choice of $Y_\bullet$. Indeed, if we choose the $Y_i$, for  $i=1,\ldots, n-1$, to 
be complete intersections of generic divisors in the linear system of the 
line bundle $L_0^q$ on $Y$, and the point $Y_n \in Y_{n-1}$ generically, 
then the Okounkov body of the divisor of $L_0^q$ is, up to scaling, a 
standard simplex (cf. \cite{sep}).
An interesting question is now for which other divisors on $Y$ a flag $Y_\bullet$ of this 
type gives a rational polyhedral Okounkov body. A natural first attempt would be to study the 
Okounkov bodies $\Delta_{Y_\bullet}(E)$ of divisors $E$ coming from the interior of the 
GIT-equivalence class $C$ of $L=\mathcal{O}_X(D)$. 

In this context, it may be very useful to know that $Y$ is a Mori dream space, since 
the pseudo-effective cone $\overline{\mbox{Eff}}(Y)$ exhibits a chamber structure 
similar to that of a smooth projective surface. It is therefore conceivable 
that the results in (\cite{SS14}), notably the construction of a Minkowski base for an Okounkov body, 
could be generalized to our setting, or even to general Mori dream spaces.
\end{rem}

\end{document}